\documentclass[12pt]{article}

\usepackage{amsfonts, amssymb, amsmath, amsthm, comment}

\theoremstyle{definition}

\theoremstyle{plain}
\newtheorem{theorem}{Theorem}[section]
\newtheorem{lemma}[theorem]{Lemma}

\theoremstyle{definition}
\theoremstyle{remark}
\newtheorem{remark}[theorem]{Remark}

\newcommand{\chitau}{\chi_{(-\tau,\tau)}}
\newcommand{\chione}{\chi_{(-1,1)}}
\newcommand{\Ltsp}{\textbf{\textit{L}}^2}
\newcommand{\LtR}{\textbf{\textit{L}}^2 (\mathbb{R})}
\newcommand{\RR}{\mathbb{R}}

\renewcommand{\leq}{\leqslant}

\renewcommand{\geq}{\geqslant}

\usepackage{url}

\makeatletter
\def\url@leostyle{%
  \@ifundefined{selectfont}{\def\UrlFont{\sf}}{\def\UrlFont{\small\ttfamily}}}
\makeatother

\begin{document}

\title{On the prolate spheroidal wave functions and Hardy's uncertainty principle}


\author{Elmar Pauwels 
\thanks{NuHAG, Faculty of Mathematics, University of Vienna,
\mbox{elmar.pauwels@univie.ac.at}} \ \ 
Maurice de Gosson
\thanks{Corresponding author, NuHAG, Faculty of Mathematics, University of Vienna, 
\mbox{maurice.de.gosson@univie.ac.at}}}

\date{}

\maketitle

\begin{abstract}
  We prove a weak version of Hardy's uncertainty principle using
  properties of the prolate spheroidal wave functions (PSWFs). We
  describe the eigenvalues of the sum of a time limiting operator and
  a band limiting operator acting on $L^2(\mathbb{R})$.  A weak
  version of Hardy's uncertainty principle follows from the asymptotic
  behavior of the largest eigenvalue as the time limit and the band
  limit approach infinity.  An asymptotic formula for this eigenvalue
  is obtained from its well-known counterpart for the prolate integral
  operator.

\end{abstract}

\section{Introduction}

Classical Hardy's uncertainty principle is formulated as follows.
\begin{theorem}
Let $a,b,M > 0$, and let $f$ be a measureable function on $\RR$ such that 
\begin{equation}
  |f(x)| \leq M \, e^{-ax^2/2} ,
\end{equation}
and
\begin{equation}
	|\hat{f}(\xi)| \leq M \, e^{-b\xi^2/2},
\end{equation}
for all $x,\xi \in \RR$. If $ab > 1$, then $f = 0$.
\end{theorem}
Several proofs of this classical theorem are known e.g.
\cite{hup_revisited, sharphardy, hardyoriginal, tao_blog}.  Typically,
they use methods of complex analysis, and rely on somewhat indirect
arguments. Our objective is to give a new and direct proof with
methods of real analysis, however only of the following weaker result.
\begin{theorem} \label{thm:hardyweak}
Let $a,b,M > 0$, and let $f$ be a measureable function on $\RR$ such that 
\begin{equation} \label{equ:boundtimedom}
  |f(x)| \leq M \, e^{-ax^2/2} 
\end{equation}
and
\begin{equation} \label{equ:boundfrequdom}
	|\hat{f}(\xi)| \leq M \, e^{-b\xi^2/2},
\end{equation}
for all $x,\xi \in \RR$. If $ab \geq 4$, then $f = 0$.
\end{theorem}

We prove this weak version of Hardy's uncertainty principle using
properties of the prolate spheroidal wave functions (PSWFs), which
appear e.g. in a solution of the concentration problem for bandlimited
functions \cite{slepian_conc}.  First, we describe the spectrum of the
sum of a time limiting operator and a band limiting operator acting on
$L^2(\mathbb{R})$. Specifically, we express the spectrum in terms of
the eigenvalues of the prolate integral operator.

Then we derive the weak version of Hardy's uncertainty principle from
the asymptotic behavior of the largest eigenvalue of the sum of the
time and band limiting operators as the time limit and the band limit
approach infinity.  An asymptotic formula for this eigenvalue is
obtained from its well-known counterpart for the prolate integral
operator.

Our approach reveals a relationship between Hardy's uncertainty
principle and the theory of bandlimited functions.

\section{Mathematical preliminaries}

The \textit{Fourier transform} of a function $f \in \LtR$ is a bounded
operator on $\LtR$ defined as follows:
\begin{equation}
  \mathcal{F} f (\xi) = 
    \frac{1}{\sqrt{2 \pi}} \lim_{N \rightarrow \infty} \int_{-N}^N e^{-i x \xi} f(x) dx,
\end{equation}
where the limit is taken in $\LtR$.
We also use the notation $\hat{f}$ for $\mathcal{F} f$.

$\mathcal{F}$ is invertible on $\LtR$, and its bounded inverse, the 
\textit{inverse Fourier transform}, is defined as follows:

\begin{equation}
  \mathcal{F}^{-1} g(x) = 
    \frac{1}{\sqrt{2 \pi}} \lim_{N \rightarrow \infty} \int_{-N}^N e^{i \xi x} g(\xi) d\xi.
\end{equation}
Consequently,
\begin{equation}
	\mathcal{F}^{-1} = \mathcal{F}^\ast,
\end{equation}
i.e. the Fourier transform is a unitary operator on $\LtR$.

A bounded operator $P$ on a Hilbert space is called
\textit{idempotent}, if
\begin{equation}
  P^2 = P.
\end{equation}
A bounded operator $P$ on a Hilbert space is called an
\textit{orthogonal projection}, if it is idempotent and Hermitian,
i.e.
\begin{equation}
  P^2 = P \quad \text{and} \quad P^\ast = P.
\end{equation}

We denote the \textit{characteristic function} of a set $E \subset \RR$
by $\chi_E$
\begin{equation} \label{equ:defchi}
  \chi_E (x) = \left\{ 
               \begin{array}{ll}
               1, & \mbox{if $x \in E$},\\
               0, & \mbox{otherwise}.
               \end{array} \right.
\end{equation}
For a fixed set $E \subset \RR$, the mapping $f \mapsto \chi_E f$ is an 
orthogonal projection on $\LtR$. We denote this projection also by $\chi_E$.
In particular, we use the notation $\chitau$, when $E = (-\tau,\tau)$,
$\tau > 0$. 


For a fixed $\omega > 0$, we define the operator $S_\omega$ as follows
\begin{equation} \label{equ:defSomega}
	S_\omega := \mathcal{F} \chi_{(-\omega,\omega)} \mathcal{F}^\ast.
\end{equation}
The operator $S_\omega$ is also an orthogonal projection on $\LtR$.

The integral kernel of $S_\omega$ is well known, namely
\begin{equation}
  S_\omega f(x) = \int_\RR \frac{\sin \omega(x-y)}{\pi (x-y)} f(y) dy.
\end{equation}
This kernel is computed as follows:
\begin{equation} \label{equ:sinckernel}
  k(x,y) = \frac{1}{2 \pi} \int_{-\omega}^\omega e^{- i x t} e^{i t y} dt
         = \frac{1}{\pi} \frac{\sin \omega(x-y)}{x-y}.
\end{equation}
Due to \eqref{equ:defSomega}, $S_\omega$ is also an orthogonal
projection on $\LtR$.

If $P$ is an orthogonal projection, then
\begin{eqnarray}
  (I - P^\ast)(I - P)
  &=& (I - P)^2 \\
  &=& I - 2P + P^2 \\
  &=& I - P.
\end{eqnarray}
Thus for every vector $f$,
\begin{eqnarray}
  \Vert (I - P) f  \Vert^2
  &=& \langle (I - P^\ast)(I - P) f,f \rangle \\
  &=& \langle (I - P) f,f \rangle. \label{equ:norminnerprod}
\end{eqnarray}

For a bounded operator $T$, $\sigma(T)$ denotes the spectrum of $T$.
We need the following well-known lemma, see \cite[Prop. 6,
p. 16]{bonsall_duncan}.
\begin{lemma} \label{lem:spectrum}
  Let $A$ and $B$ be bounded operators on a Hilbert space. For $\lambda \neq 0$,
  $\lambda \in \sigma (AB)$ if and only if $\lambda \in \sigma (BA)$.
\end{lemma}

The following lemma has a straightforward proof, which is omitted.
\begin{lemma} \label{lem:lambdaPinverse} If $P$ is an idempotent
  bounded operator on a Hilbert space and $\lambda \neq 0,1$, then
  \begin{equation}
    (\lambda I  - P)^{-1} = \frac{1}{\lambda} I + \frac{1}{\lambda (\lambda-1)} P.
  \end{equation}
\end{lemma}

\section{Spectrum of $\chitau$ + $S_\omega$}

In this section, we describe the spectrum of the operator $\chitau + S_\omega$
on $\LtR$, where $\tau,\omega > 0$.
 
For a fixed $c > 0$, the integral operator on $\Ltsp(-1,1)$ with the kernel 
\begin{equation} 
  \frac{\sin c(x-y)}{\pi (x-y)} 
\end{equation}
has eigenvalues $\lambda_0 > \lambda_1 > \dots > 0$ \cite{slepian_conc, slepian_sonnenblick}. 
The eigenfunctions are the PSWFs, and the eigenvalues obey certain asymptotic formulas. 
We are only interested in the largest eigenvalue $\lambda_0$ which has the following 
asymptotics \cite{slepian_sonnenblick}
\begin{equation} \label{equ:lambda0}
  \lambda_0 = 
  1 - 4 \sqrt{\pi} \sqrt{c} \,e^{-2 c} \left( 1 + \mathcal{O} \left(\frac{1}{c} \right) \right),
\end{equation}
when $c \rightarrow \infty$.

We show that the eigenvalues of $T = \chitau + S_\omega$, acting on $\LtR$, can be 
expressed in terms of those of the operator with kernel
\eqref{equ:sinckernel}, acting not on $\LtR$, but rather on $\Ltsp(-\tau,\tau)$.

To indicate the dependence on $c$, we write $\lambda_n(c)$.

\begin{theorem} \label{lem:eigenvalue}
  Let $\tau,\omega > 0$ and let $T = \chitau + S_\omega$. 
  If $\lambda \in \sigma(T)$ and $\lambda \neq 0,1$, then
  \begin{equation}
    (\lambda - 1)^2 = \lambda_n (\omega \tau)
  \end{equation}
  for some $n$. 
\end{theorem}
\begin{proof}
  In this proof we can assume that $\tau = 1$. The general case follows by a linear
  change of variables.
   
  It follows from the assumptions that $\lambda I - \chione - S_\omega$ is singular, 
  and so is the operator
  \begin{eqnarray}
    (\lambda I - S_\omega)^{-1} (\lambda I - \chione - S_\omega) 
    &=& I - (\lambda I - S_\omega)^{-1} \chione \\
    &=& I - \left( \frac{1}{\lambda} I + \frac{1}{\lambda (1-\lambda)} S_\omega \right) \chione,
  \end{eqnarray}
  where we used Lemma \ref{lem:lambdaPinverse} for $(\lambda I - S_\omega)^{-1}$.
  Thus
  \begin{eqnarray}
    \lambda 
    &\in& \sigma \left( \chione + \frac{1}{\lambda - 1} S_\omega \chione \right) \\
    &=& \sigma \left( \chione + \frac{1}{\lambda - 1} S_\omega \chione^2 \right).    
  \end{eqnarray}
  We use Lemma \ref{lem:spectrum} for the operators $\chione$ and $I + \frac{1}{\lambda - 1} S_\omega \chione$
  and the assumption that $\lambda \neq 0$, to conclude that 
  \begin{equation}
    \lambda \in \sigma \left( \chione + \frac{1}{\lambda - 1} \chione S_\omega \chione \right).
  \end{equation}    
  It follows that the operator
  \begin{equation}
    \lambda I - \chione - \frac{1}{\lambda - 1} \chione S_\omega \chione
  \end{equation}
  is singular, and so is
  \begin{equation}
    I - (\lambda I - \chione)^{-1} \frac{1}{\lambda - 1} \chione S_\omega \chione.
  \end{equation}
  The operator $\chione S_\omega \chione$ is Hilbert-Schmidt, and therefore compact,
  and so is $(\lambda I - \chione)^{-1} \chione S_\omega \chione$.
  
  From the spectral theory of compact operators, it follows that $1$ is an eigenvalue
  of 
  \begin{equation}
    (\lambda I - \chione)^{-1} \frac{1}{\lambda - 1} \chione S_\omega \chione,
  \end{equation}
  i.e. there is an $f \in \LtR$, $f \neq 0$ such that
  \begin{equation}
    f - (\lambda I - \chione)^{-1} \frac{1}{\lambda - 1} \chione S_\omega \chione f = 0
  \end{equation}
  or 
  \begin{equation} \label{equ:lambdaf}
    \lambda f - \chione f - \frac{1}{\lambda - 1} \chione S_\omega \chione f = 0.
  \end{equation}
  Using \eqref{equ:lambdaf}, we note that $\chione f \neq 0$ in $\LtR$, 
  because $\lambda f \neq 0$ in $\LtR$.
  We now multiply \eqref{equ:lambdaf} by $\chione$ on the left to get
  \begin{equation}
    \lambda \chione f - \chione f - \frac{1}{\lambda - 1} \chione S_\omega \chione f = 0
  \end{equation} 
  and, consequently
  \begin{equation}
    (\lambda - 1)^2 \chione f - \chione S_\omega \chione f = 0.
  \end{equation} 
  Thus $\chione f$ is an eigenfunction of the operator with kernel 
  $\frac{\sin \omega (x-y)}{\pi (x-y)}$ on $\Ltsp (-1,1)$ with 
  eigenvalue $(\lambda - 1)^2$.

  Thus
  \begin{equation}
    (\lambda - 1)^2 = \lambda_n (\omega)
  \end{equation}
  for some $n$.
\end{proof}

\begin{remark}
  Only Theorem \ref{lem:eigenvalue} is used in the proof of Theorem \ref{thm:hardyweak}.
  However, we devote the rest of this section to a complete description of the spectrum
  of the operator $T = \chitau + S_\omega$.
  
  In this proof we can again assume that $\tau = 1$. The general case follows by a linear
  change of variables.
  
  We can show that if $(\lambda - 1)^2 = \lambda_n (\omega)$, then there exists an 
  eigenfunction for $T$ with eigenvalue $\lambda$.
    
  It follows from Slepian's theory, that $0 < \lambda_n (\omega) < 1$. We write
  \begin{equation} \label{equ:fpsi}
    f = \psi_n + (\lambda - 1) \widetilde{\psi}_n,
  \end{equation}
  where $\psi_n$ is the $n$th PSWF and $\widetilde{\psi}_n$
  is the extension of $\psi_n$ to $\RR$, i.e.
  \begin{equation}
    S_\omega \psi_n = \lambda_n (\omega) \widetilde{\psi}_n
  \end{equation}
  and
  \begin{equation}
    \chione \widetilde{\psi}_n = \psi_n.
  \end{equation}
  We multiply \eqref{equ:fpsi} by $\chione$ to get
  \begin{equation}
    \chione f 
    = \psi_n + (\lambda - 1) \psi_n
    = \lambda \psi_n.
  \end{equation}
  Moreover,
  \begin{eqnarray}
    (\lambda I - S_\omega)^{-1} \chione f
    &=& \left( \frac{1}{\lambda} I + \frac{1}{\lambda (\lambda - 1)} S_\omega \right) \lambda \psi_n \\
    &=& \psi_n + \frac{1}{\lambda - 1} \lambda_n \widetilde{\psi}_n \\
    &=& \psi_n + (\lambda - 1) \widetilde{\psi}_n \\
    &=& f,
  \end{eqnarray}
  where we used the assumption that $\lambda_n = (\lambda - 1)^2$.
  
  Finally,
  \begin{equation}
    \chione f = (\lambda I - S_\omega) f
  \end{equation}
  or
  \begin{equation}
    \chione f + S_\omega f = \lambda f.
  \end{equation}
  Thus all numbers of the form $1 \pm \sqrt{\lambda_n (\omega)}, 
  n = 0,1,\dots$ are eigenvalues of $T$.

  The point $\lambda = 1$ is also in the spectrum $\sigma(T)$ as an accumulation
  point of the eigenvalues.

  It remains to consider $\lambda = 0$. To observe that $T$ is singular, we consider
  the sequence of functions
  \begin{equation}
    f_n (x) = e^{i n x} e^{- (x - n)^2}
  \end{equation}
  where $x \in \RR$. It is clear that
  \begin{equation}
   \Vert f_n \Vert = \Vert f_0 \Vert > 0,
  \end{equation}
  for $n = 0,1,\dots$, but
  \begin{equation}
    |T f_n| \rightarrow 0
  \end{equation}
  in $\LtR$. 

  Thus we have shown that $\sigma(T)$ consists of the eigenvalues of $T$ of the
  form $1 \pm \sqrt{\lambda_n (\omega)}, n = 0,1,\dots$, and two additional points
  $\lambda = 0$ and $\lambda = 1$.

  The spectrum of the sum of two orthogonal projections was described
  in \cite{bjorstad} in a somewhat different setting.

\end{remark}

\section{Proof of the Theorem \ref{thm:hardyweak}}
In this section, we present the proof of Theorem \ref{thm:hardyweak}.

\begin{proof}[Proof of Theorem \ref{thm:hardyweak}]

In this proof, we assume that
\begin{equation} \label{equ:abtwo}
  a = b = 2.
\end{equation}
The general case follows by a linear change of variables.

For a fixed $\tau > 0$, we consider the restriction $\chitau f$ of $f$ to the interval
$(-\tau,\tau)$. The decay of $f$ at infinity in \eqref{equ:boundtimedom} gives an
estimate on $f - \chitau f$ in the $\Ltsp$-norm. Specifically,
\begin{eqnarray}
  \Vert f - \chitau f \Vert^2
  &=& \int_{|x|> \tau} |f(x)|^2 dx \\
  &\leq& M^2 \int_{|x|>\tau} (e^{-a x^2/2})^2 dx \\
  &=& 2 M^2 \int_\tau^\infty e^{-a x^2} dx \\
  &\leq& 2 M^2 \int_\tau^\infty \frac{x}{\tau} e^{-a x^2} dx \\
  &=& \frac{M^2}{a \tau} e^{- a\tau^2}. \label{equ:normestchi}
\end{eqnarray}
Similarly, \eqref{equ:boundfrequdom} implies that for a fixed $\omega > 0$,
\begin{equation} \label{equ:normestS}
	\Vert f - S_\omega f \Vert^2 
	\leq \frac{M^2}{b \omega} e^{- b\omega^2}.
\end{equation}

Setting $\tau = \omega$, using \eqref{equ:norminnerprod} and combining \eqref{equ:abtwo},
\eqref{equ:normestchi} and \eqref{equ:normestS} gives
\begin{eqnarray}
   \lefteqn{\Vert f - \chitau f \Vert^2 + \Vert f - S_\omega f \Vert^2 =} \\
   &=&  \langle (I - \chitau) f,f \rangle + \langle (I - S_\omega) f,f \rangle \\
   &=& \langle (2I - \chitau - S_\omega) f,f \rangle \\
   &\leq& \frac{M^2}{\omega} e^{- 2\omega^2}. \label{equ:upperbound}
\end{eqnarray}

The operator $T' = 2I - \chitau - S_\omega$ is Hermitian. According to
Theorem \ref{lem:eigenvalue}, its smallest eigenvalue $\lambda_{min}$ 
satisfies
\begin{equation}
  \lambda_{min}
  \geq 2 - (1 + \sqrt{\lambda_0}) = 1 - \sqrt{\lambda_0}.
\end{equation}
Consequently,
\begin{eqnarray}
  (1 - \sqrt{\lambda_0}) \Vert f \Vert^2 
  &\leq& \lambda_{min} \Vert f \Vert^2 \\
  &\leq& \langle (2I - \chitau - S_\omega) f,f \rangle \\
	&\leq& \frac{M^2}{\omega} e^{- 2\omega^2}. \label{equ:bounds}
\end{eqnarray}
The eigenvalue $\lambda_0$ satisfies \eqref{equ:lambda0}. Thus since 
$c = \omega \tau = \omega^2$, we get
\begin{equation} \label{equ:lambda0omega}
\lambda_0 = 1 - 4 \sqrt{\pi} \,\omega \,e^{-2 \omega^2} \left( 1 + \mathcal{O}
\left(\frac{1}{\omega^2} \right) \right).
\end{equation}
We recall the elementary formula
\begin{equation} \label{equ:sqrtoneminusx}
  \sqrt{1-x} = 1 - x/2 + \mathcal{O}(x^2)
\end{equation}
as $x \rightarrow 0$.
Substituting \eqref{equ:lambda0omega} and \eqref{equ:sqrtoneminusx} into \eqref{equ:bounds} 
we obtain
\begin{equation}
  2 \sqrt{\pi} \,\omega \,e^{-2 \omega^2}  \Vert f \Vert^2 \left( 1 + \mathcal{O} \left(\frac{1}{\omega^2} \right) \right) 
  \leq \frac{M^2}{\omega} e^{- 2\omega^2}.
\end{equation}
Letting $\omega \rightarrow \infty$, we deduce that $\Vert f \Vert = 0$.
\end{proof}

\subsection{Alternative proof}
A reviewer of this paper has remarked that an alternative proof is
possible based on the following result proved in
\cite[p. 68]{landau_pollak}.
\begin{theorem}
 If $\Vert f \Vert = 1$,
%
\begin{equation}
\alpha = \left(\int_{-\frac{T}{2}}^\frac{T}{2} |f(t)|^2 dt \right)^{\frac12},
\end{equation}
\begin{equation}
\beta = \left(\int_{-\Omega}^\Omega |\hat{f}(\xi)|^2 d\xi \right)^{\frac12},
\end{equation}
then
\begin{equation} \label{equ:landau}
   \arccos \alpha + \arccos \beta \geq \arccos
   \sqrt{\lambda_0\left({\textstyle \frac12} \Omega T\right)}.
\end{equation}
\end{theorem}

We present an outline of an alternative proof. Let us assume that
$\Vert f \Vert = 1$. It follows from \eqref{equ:abtwo} and
\eqref{equ:normestchi} that
\begin{equation} \label{equ:upperboundchif0}
	\Vert \chitau f \Vert^2 = \Vert f \Vert^2 - \Vert f - \chitau
        f \Vert^2 \geq 1 - \frac{M^2}{2 \tau} \, e^{- 2\tau^2}.
\end{equation}
Consequently,  for all sufficiently large $\tau$'s,
\begin{equation}
  \Vert \chitau f \Vert \geq \left(1 - \frac{M^2}{2 \tau} \, e^{-2\tau^2} \right)^{\frac12}
  \geq 1 - \frac{M^2}{2 \tau} \, e^{- 2\tau^2},
\end{equation}
and 
\begin{equation}  \label{equ:upperboundchif}
  \arccos\Vert \chitau f \Vert \leq 
  \arccos \left(1-\frac{M^2}{2\tau} \, e^{- 2\tau^2}\right).
\end{equation}
We recall that as $x \rightarrow 0^+$,
\begin{equation} \label{equ:arccos}
	\arccos(1-x) = \sqrt{2x} \left( 1 + \mathcal{O}(x) \right).  
\end{equation}
Combining \eqref{equ:upperboundchif} and \eqref{equ:arccos}, we
conclude that for every sufficiently large $\tau$
\begin{equation} \label{equ:upperboundstau}
\arccos\Vert \chitau f \Vert \leq \frac{2M}{\sqrt\tau} \, e^{- \tau^2}.
\end{equation}
Similarly,  for every sufficiently large $\omega$
\begin{equation}\label{equ:upperboundsomega}
\arccos\Vert S_\omega f \Vert \leq \frac{2M}{\sqrt\omega} \, e^{-\omega^2}.
\end{equation}
Combining \eqref{equ:upperboundstau} and \eqref{equ:upperboundsomega},
and setting $\tau = \omega$, we obtain
\begin{equation}
	\arccos\Vert \chi_{(-\omega,\omega)} f \Vert + \arccos\Vert S_\omega f \Vert 
	\leq \frac{4M}{\sqrt\omega} \, e^{-\omega^2}.
\end{equation}
Setting $\Omega = \omega$ and $T = 2\omega$ in \eqref{equ:landau}, we
obtain
\begin{equation} 
	\arccos \sqrt{\lambda_0(\omega^2)} \leq
	\arccos\Vert \chi_{(-\omega,\omega)} f \Vert + \arccos\Vert S_\omega f \Vert.
\end{equation}
Consequently,
\begin{equation} \label{equ:final1}
	\arccos \sqrt{\lambda_0(\omega^2)} \leq
        \frac{4M}{\sqrt\omega} \, e^{-\omega^2}.
\end{equation}
Substituting \eqref{equ:lambda0omega} into \eqref{equ:sqrtoneminusx},
we obtain
\begin{equation} \label{equ:sqrtlambda}
\sqrt{\lambda_0(\omega^2)} = 1 - 2\sqrt{\pi} \,\omega \,e^{-2
\omega^2}  \left(1 +  \mathcal{O}\left(\frac{1}{\omega^2} \right) \right) .
\end{equation}
Substituting \eqref{equ:sqrtlambda} into \eqref{equ:arccos}, we obtain
\begin{equation} \label{equ:acoslambda}
\arccos \sqrt{\lambda_0(\omega^2)} = 
2\sqrt[4]{\pi} \,\sqrt\omega \,e^{-\omega^2}  
\left(1 +  \mathcal{O}\left(\frac{1}{\omega^2} \right) \right) .
\end{equation}
Combining \eqref{equ:final1} and \eqref{equ:acoslambda}, we arrive at
a contradiction
\begin{equation}
2\sqrt[4]{\pi} \,\sqrt\omega \,e^{-\omega^2}  
\left(1 +  \mathcal{O}\left(\frac{1}{\omega^2} \right) \right) 
\leq \frac{4M}{\sqrt\omega} \, e^{-\omega^2}.
\end{equation}

Our proof of Theorem~\ref{thm:hardyweak} uses some techniques similar
to those in \cite{landau_pollak}, e.g. a linear combination of the
time and the frequency limiting operators is already considered in
\cite[equation (6)]{landau_pollak}.

\subsection*{Acknowledgements}

The authors thank the reviewers for their helpful comments and
suggestions, which have greatly improved this paper. The authors are
supported by the FWF grants S10602-N13 and P23902-N13.

\bibliography{hardyup}
\bibliographystyle{abbrv}

\end{document}